\documentclass[subeqn, letterpaper, reqno, 11pt]{amsart}

\usepackage{hyperref}

\usepackage{latexsym}
\usepackage{amssymb}
\usepackage{amsmath}
\usepackage{amsfonts}
\usepackage[mathscr]{eucal}
\usepackage{times}

\usepackage{amscd}

\usepackage[numbers, sort&compress]{natbib}

\usepackage{psfrag}

\usepackage{amsthm}
\usepackage[all,cmtip]{xy}

\usepackage{comment}

\input{xypic}

\usepackage{tikz}
\usetikzlibrary{backgrounds,calc}
\usetikzlibrary{calc,arrows}

\newtheorem{thm}{Theorem}[section]

\newtheorem{lemma}[thm]{Lemma}

\newtheorem{prop}[thm]{Proposition}

\theoremstyle{definition}
\newtheorem{definition}[thm]{Definition}

\newtheorem{question}[thm]{Question}

\newtheorem*{ack}{Acknowledgements}

\numberwithin{equation}{section}

\newcommand{\RR}{\mathbb{R}}

\DeclareMathOperator{\Diam}{diam}

\DeclareMathOperator{\Rad}{rad}

\begin{document}



\title[Rigidity of the parallel postulate]{Riemannian Rigidity of the Parallel Postulate in Total Curvature}


\author[J.~Ge]{Jian Ge*}
\address[Ge]{Beijing International Center for Mathematical Research, Peking University, Beijing 100871, P. R. China.}
\email{jge@math.pku.edu.cn}
\thanks{*Partially supported by the Recruitment Program for Young Professionals.}


\author[L.~Guijarro]{Luis Guijarro**}
\address[Guijarro]{ Department of Mathematics, Universidad Aut\'onoma de Madrid, and ICMAT CSIC-UAM-UCM-UC3M, Spain.}
\curraddr{}
\email{luis.guijarro@uam.es}
\thanks{**Supported by research grants MTM2014-57769-3-P
(MINECO) and ICMAT Severo Ochoa project SEV-2015-0554 (MINECO)}

\author[P.~Sol\'orzano]{Pedro Sol\'orzano***}
\address[Sol\'orzano]{Instituto de Matem\'aticas--Oaxaca, Universidad Nacional Aut\'onoma de M\'exico, Oaxaca de Ju\'arez, Mexico.}
\curraddr{}
\email{pedro.solorzano@matem.unam.mx}
\thanks{***Supported by Mexico's {\em C\'atedras CONACYT} program as a permanent fellow. }

\date{\today}




\begin{abstract}
We study metrics in $\RR^2$ without conjugate points and with total curvature, answering, for this case, a question posed in \citet{BK1991} and \citet{BE2013}. Namely, we prove that the Euclidean plane is the only Riemannian surface free of conjugate points with total curvature that satisfies Playfair's version of the parallel postulate.
\end{abstract}
\maketitle




\section{Introduction and results}

The \emph {Parallel Postulate (PA)} or, more precisely, \emph {Playfair's version of Euclid's fifth postulate} says that
\begin{quote}
``Given any straight line in the plane and a point not lying on it, there exists one and only one straight line which passes through that point and never intersects the first line.''
\end{quote}
 It is now well known that the Parallel Postulate is an axiom that does not necessarily hold in non-Euclidean Geometry.

 For a Riemannian manifold $M$, the proposition 1.7 of \citet{BK1991} states that a geodesic $\gamma: \mathbb R\to M$ has no conjugate points if and only if there exists a geodesic from every point $p\not\in\gamma(\mathbb R)$  that doesn't intersect  $\gamma(\mathbb R)$.  In particular, for any $s, t\in \mathbb R$, we have that $d(\gamma(t), \gamma(s))=|t-s|$, where $d(\cdot\,,\,\cdot)$ is the distance induced by the Riemannian metric on $M$. A geodesic that satisfies this last condition is called a {\em line}.  Therefore, for a Riemannian surface, i.e. 2-dimensional Riemannian manifold, we will say that
 $M$ satisfies the {\it Parallel Axiom}, if
\begin{definition}[Parallel Axiom]\label{def:PA}
For every geodesic $\gamma$ on $M$ and $p\not\in\gamma(\mathbb R)$, there exists a unique geodesic through $p$ that does not intersect $\gamma(\mathbb R)$.
\end{definition}
Any such Riemannian surface would then have no conjugate points and thus every geodesic would actually be a line. Clearly the Euclidean plane $\mathbb R^2$ with the flat metric satisfies the parallel axiom. It is an open question whether the flat $\mathbb R^2$ is the only Riemannian manifold homeomorphic to $\mathbb R^{2}$ with such property.

\begin{question}[\citet{BE2013}, \citet{BK1991}, \citet{Cr2006}]
Is the Euclidean metric the only complete metric on the plane with
no conjugate points such that for every line $\ell$ (complete geodesic) and point $p$ not
on $\ell$ there is a unique line through $p$ parallel to (i.e. not intersecting) $\ell$?
\end{question}

In this note, we give an affirmative answer to this question with the additional assumption that $M$ has total curvature.

\begin{thm}\label{thm:Main}
Let $M$ be a Riemannian plane satisfying the Parallel Axiom and admiting total curvature. Then $M$ is isometric to the flat Euclidean plane $\mathbb R^2$.
\end{thm}

%
%
%
%

The paper is organized as follows. Section \ref{sect:ideal boundary} collects the main properties of the ideal boundary that are necessary in th rest of the article. We refer the reader to the excelent survey \cite{Shi1996} for more details on this topic. The proof of Theorem \ref{thm:Main} appears in section \ref{sect:proof theorem PA}. In Section \ref{sect:Euclid} we give a version of Theorem \ref{thm:Main} when Euclid's fifth postulate is written in its original version (see Theorem \ref{thm:Euclid original}.


\begin{ack}
 The authors would like to thank the Department of Mathematics of the Universidad Aut\'o\-noma de Madrid, where the present work was initiated. The third named author wishes to thank Victor Bangert for bringing this problem to his attention, as well as Wolfgang Ziller and IMPA for their hospitality during the meeting where this happened.
\end{ack}


\section{Total curvature and ideal boundary}\label{sect:ideal boundary}
In this section we recall the definitions of total curvature and of the ideal boundary. We will be brief, recommending the article \cite{Shi1996} for further reading if necessary.

Let $K: M\to \mathbb R$ be the Gaussian curvature of a Riemannian surface $M$, and let $d\mu$ be its associated Riemannian density. Define $K^+=\max\{K, 0\}$ and $K^-=\min\{K, 0\}$. Thus we can define \emph{the total positive curvature} $c_+(M)$ and \emph{total negative curvature} $c_-(M)$ by
\[
c_+(M):=\int_M K^+d\mu,\
\quad
\ c_-(M):=\int_M K^-d\mu.
\]
We will say that the total curvature $c(M)$ of $M$ exists, or that {\em $M$ admits total curvature}, if at least one of $c_+(M)$ or $c_-(M)$ is finite, and we write in that case
$$
c(M):=c_+(M)+c_-(M).
$$
By theorems of \citet{MR1556908} and \citet{MR0224030}, the total curvature exists if and only if $c_+(M)$ is finite. Through this paper, we will assume that $M$ is homeomorphic to $\mathbb R^2$, with a Riemannian metric $g$ such that the total curvature of $(M, g)$ exists.

Under these assumptions, we can define the ideal boundary  $M(\infty)$ in several different ways \cite{Shi1996}. We choose among these the most convenient for the proof of the results mentioned in the introduction.

 Let $\Gamma(p)$ be the set of rays emanating from a fixed point $p\in M$. Denote by $d_t(\cdot\,,\,\cdot)$ the path metric on the distance sphere $S(t):=\partial B(p, t)$, i.e. $d_t(a, b)$ is the length of a shortest arc connecting $a$ to $b$ on $S(t)$.
   For two rays $\gamma$ and $\sigma$ with $p=\sigma(0)=\gamma(0)$, Shioya proved that the limit
 \begin{equation}\label{eq:Ideal=Circle}
d_\infty(\gamma(\infty), \sigma(\infty))=\lim_{t\to \infty}\frac{d_t(\gamma(t), \sigma(t))}{t}
\end{equation}
  exists. The generalized metric $d_\infty: \Gamma(p)\to [0,\infty]$ is known as the \emph{Tits metric} on $\Gamma(p)$.
The ideal boundary of $M$ is defined as
\[
M(\infty):=\Gamma(p)/\sim,
\]
where we identify $\gamma\sim \sigma$ if and only if $d_\infty(\gamma(\infty), \sigma(\infty))=0$. We will keep on using $d_\infty$ for the distance induced in this quotient.

Next we include some of the properties of the ideal boundary necessary for the rest of the paper.
The first one determines $M(\infty)$:
\begin{prop}[\citet{Shi1996}, 1.1.3]\label{prop:IdealBoundary}
The ideal boundary $(M(\infty), d_\infty)$ is an inner metric space, lying within one of the following cases:
\begin{enumerate}
\item\label{case:(1)} a single point,
\item\label{case:(2)} a closed circle with finite length,
\item\label{case:(3)} at most continuum disjoint union of closed intervals, which might also be a single point or infinity intervals.
\end{enumerate}
\end{prop}
For a {\it line} $\gamma: (-\infty, \infty)\to M$, we call $\gamma^+(t)=\gamma(t)$ and $\gamma^-(t)=\gamma(-t)$ the positive and negative rays associated to the line $\gamma$, and we denote by $\gamma(-\infty)$ (resp. $\gamma(\infty)$ ) the associated ideal boundary point of the ray $\gamma^-$ (resp. $\gamma^+$). If $\gamma, \sigma$ are two rays, we will say that \emph{$\gamma$ is asymptotic to $\sigma$} if the geodesic segments $\overline{\gamma(0)\sigma(t)}$ converge to $\gamma$;
This will be denoted as $\sigma\sim\gamma$.

When one ray is asymptotic to  another, we cannot tell them apart at infinity; in fact by \cite[Theorem 1.1.2]{Shi1996},  we have $d_{\infty}(\sigma(\infty), \gamma(\infty))=0$ if $\sigma\sim\gamma$.

We also need the following
\begin{prop}[\citet{Shi1996}, Theorem 1.1.5]\label{prop:PropOfLine}
For any line $\gamma$, we have $$d_\infty(\gamma(-\infty), \gamma(\infty))\ge \pi.$$
Moreover, for any $x, y\in M(\infty)$, with $d_\infty(x, y)> \pi$, there exists a line $\sigma$ with $\sigma(-\infty)=x$ and $\sigma(\infty)=y$.
\end{prop}

In the paper, we will also need the construction of the line $\sigma$ in Shioya's paper, so we recall briefly its construction. Let $\alpha$ and $\beta$ be two rays such that $\alpha(\infty)=x$ and $\beta(\infty)=y$. Pick $t_{n}\to\infty$ as $n\to \infty$, consider the geodesic segments $\{\gamma_{n}=\overline{\alpha(t_{n})\beta(t_{n})}\}$. By letting $n\to \infty$, $\gamma_{n}$ will converge to a line $\sigma$ when $d_{\infty}(x,y)>\pi$ by the Gauss-Bonnet formula.

\section{Proof of Theorem 1.1}\label{sect:proof theorem PA}
Under the assumption of the Parallel Axiom, it is easy to see that $M$ has no conjugate points (see \citet[Proposition 1.7]{BK1991}), hence every geodesic can be extended to a line and every two non identical geodesics cannot intersect in more than one point. We will need the following easy yet very useful observation whose proof is trivial.

\begin{lemma}[Convexity Lemma]\label{lem:convex}
Any line $\gamma$ separates $M$ into two totally convex sets. Hence if $\Omega$ is a locally convex domain bounded by finitely many geodesics, then $\Omega$ is convex.
\end{lemma}
\begin{prop}\label{prop:biAsymOfLine}
Let $\sigma$ and $\gamma$ be two lines parallel to each other. Then $\sigma$ and $\gamma$ are also bi-asymptotic to each other.  That is that, after a possible change of the variable $t\to -t$ for one of the two lines, both $\sigma^+\sim \gamma^+$ and $\sigma^-\sim \gamma^-$ hold.
\end{prop}
\begin{proof}
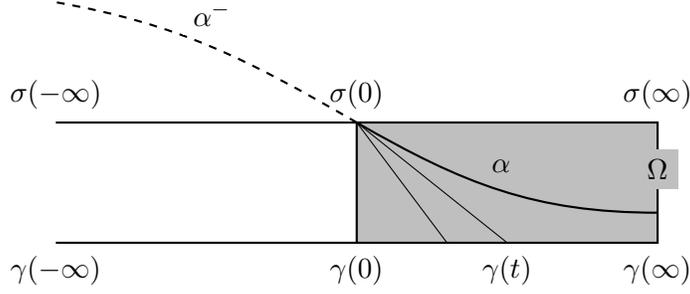
\begin{figure}
\begin{tikzpicture}[scale=4]

\draw [thick](-1, 0.4)node[above=1pt, fill=white]{$\sigma(-\infty)$} -- (0,0.4)node[above=1pt, fill=white]{$\sigma(0)$} -- (1, 0.4)node[above=1pt, fill=white]{$\sigma(\infty)$};

\draw [thick](-1, 0)node[below=1pt, fill=white]{$\gamma(-\infty)$} -- (0,0)node[below=1pt, fill=white]{$\gamma(0)$} -- (1, 0)node[below=1pt, fill=white]{$\gamma(\infty)$};

\draw [thick, draw=black, fill=lightgray](0,0) -- (0,0.4) -- (1,0.4)node[below=10pt, fill=lightgray]{$\Omega$}--(1,0)-- (0,0);
\draw (0,0.4)--(0.5, 0) node[below=1pt](1){$\gamma(t)$};
\draw (0,0.4)--(0.3, 0);
\draw [thick](0,0.4) to[out=-30, in=180] node[above=3pt, fill=lightgray](2){$\alpha$}(1, 0.1);
\draw [thick, dashed](0,0.4) to[out=150, in=350] node[above=4pt, fill=white]{$\alpha^{-}$}(-1, 0.8);
\end{tikzpicture}
\caption{Bi-asymptoticity of lines}
\label{FIG:biasym}
\end{figure}
Denote by $\Omega$ the region bounded by the rays $\sigma^{+}$, $\gamma^{+}$ and the geodesic passing trough $\gamma(0)\sigma(0)$, see \ref{FIG:biasym}. By \ref{lem:convex}, $\Omega$ is convex. Suppose the segments $\overline{\sigma(0)\gamma(t)}$ converges to a ray $\alpha$ as $t\to +\infty$. By the convexity of $\Omega$, we know that $\overline{\sigma(0)\gamma(t)}\subset \Omega$, hence $\alpha\subset \Omega$. Hence the angle $\theta$, between $\alpha$ and $\sigma^{+}$, satisfies $0\le \theta <\pi/2$. Then we can extend $\alpha$ by the ray $\alpha^{-}$ to a line, which we still denote by $\alpha$. Since $\theta\ge 0$, the ray $\alpha^{-}$ does not intersect $\gamma$. By the uniqueness of the parallel line through $\sigma(0)$ of the line $\gamma$, we see that $\alpha(t)=\sigma(t)$ for any $t$. The other asymptotic can be shown similarly.
\end{proof}
The previous result is mentioned without a proof by \citet{MR0075623}. Now we are ready to state and proof our key observation.

\begin{lemma}[Key Lemma]\label{lem:DiamEst}
The ideal boundary $M(\infty)$ is a closed circle with length $2\pi$.
\end{lemma}
\begin{proof}
Clearly the existence of a line implies the ideal boundary is not a point by Proposition \ref{prop:PropOfLine}. In fact this also implies that $\Diam(M(\infty))\ge \pi$. Hence in the list of Proposition \ref{prop:IdealBoundary}, only (\ref{case:(2)}) and (\ref{case:(3)}) could happen.

\noindent\textit{Case 1:} Suppose $M(\infty)$ is a closed circle of length $2\ell$, with $\ell >\pi$. For any line $\gamma$, we can choose a point $x\in M(\infty)$ such that $d_\infty(\gamma(-\infty), x)>\pi$ with $x\ne \gamma(\infty)$. Let $\alpha$ be a ray such that $\alpha(0)=\gamma(0)$ and $\alpha(\infty)=x$. Denote by $\Delta$ the convex sector in $M$ bounded by the rays $\gamma^{-}$ and $\alpha$ (see  \ref{FIG:tripod}). By Lemma \ref{lem:convex}, $\Delta$ is totally convex. Hence by the construction in the proof of Proposition \ref{prop:PropOfLine}, there exists a line $\sigma$ connecting $\gamma(-\infty)$ to $\alpha(\infty)$, such that
\[
\sigma(\mathbb R)\subset \Delta,
\]
by convexity of $\Delta$. By the definition, $\sigma$ is parallel to $\gamma$ and does not intersect $\alpha$. Clearly the convexity of $\Delta$ implies that the line $\alpha$ extending the ray $\alpha$, does not intersect $\sigma$. Hence we have two lines, $\alpha$ and $\gamma$, passing through $\gamma(0)$, and parallel to $\sigma$. This is a  contradiction to the Parallel Axiom. Thus if $M(\infty)$ is a circle, it necessarily has length $2\pi$.

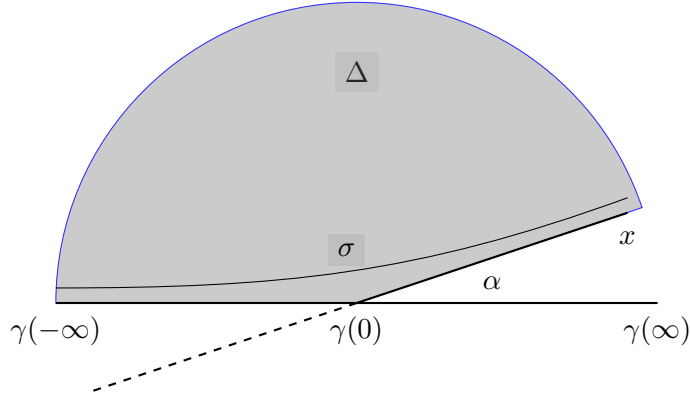
\begin{figure}
\begin{tikzpicture}[scale=4]
\filldraw[fill=lightgray, draw=blue, opacity=0.8] (0,0) -- (-1, 0)arc (180:18.5:1) -- (0,0)node[above=80pt, fill=lightgray]{$\Delta$};
\draw [thick](-1, 0)node[below=1pt, fill=white]{$\gamma(-\infty)$} -- (0,0)node[below=1pt, fill=white]{$\gamma(0)$} -- (1, 0)node[below=1pt, fill=white]{$\gamma(\infty)$};
\draw [thick](0, 0) -- node[below=3pt, fill=white]{$\alpha$}(0.9, 0.3)node[below=3pt, fill=white]{$x$};
\draw (-1,0.05) to[out=0, in=200] node[above=2pt, fill=lightgray]{$\sigma$}(0.9, 0.35);
\draw [thick, dashed](0,0) -- (-0.9, -0.3);

\end{tikzpicture}
\caption{The construction}
\label{FIG:tripod}
\end{figure}

\noindent\textit{Case 2:} Suppose that $M(\infty)$ contains  at least two connected components. Choose $z_1$, $z_2$, points in $M(\infty)$ lying in different components, hence $d_{\infty}(z_{1}, z_{2})=\infty$. By Proposition \ref{prop:PropOfLine} there is a line $\gamma$ connecting $z_{1}$ to $z_{2}$. Let $\alpha$ be a ray emanating from $\gamma(0)$ but that does not coincide with either $\gamma^{+}$ or $\gamma^{-}$. Since
$$\infty=d_{\infty}(\gamma(-\infty), \gamma(\infty))\le d_{\infty}(\gamma(-\infty), \alpha(\infty))+d_{\infty}(\alpha(\infty), \gamma(\infty)),$$
we can assume $d_{\infty}(\gamma(-\infty), \alpha(\infty))=\infty$. By the same argument as in Case 1 above, there exists a line $\sigma$ constructed from the limit of line segments $\overline{\gamma(-t)\alpha(t)}$,  hence $\sigma$ does not intersect either $\gamma$ or $\alpha$; which contradicts again the uniqueness of parallels through the point $\gamma(0)$.

\noindent\textit{Case 3:} Suppose $M(\infty)$ is a connected interval of length $\ell$. Since for any line $\gamma$, we have $d_{\infty}(\gamma(-\infty), \gamma(\infty))\geq \pi$, it implies that the radius of $M(\infty)$ satisfies
\[
\Rad(M(\infty))\ge \pi.
\]
Hence $\ell\ge 2\pi$. We denote $M(\infty)$ by $[0, \ell]$. Choose $z_{1}=0$, $z_{2}=1.5\pi$. Let $\gamma$ be a line connecting $z_{1}$ to $z_{2}$. Let $\alpha$ be a ray from $\gamma(0)$ to $z_{3}=1.2\pi\in M(\infty)$. Again, by the same construction as in Case 1, we get a line $\sigma$ from $z_{1}$ to $z_{3}$ that does not intersect $\gamma$ and $\alpha$, which contradicts once again the uniqueness of parallels.

Combining the three cases, we obtain that $M(\infty)$ is a circle of length $2\pi$.
\end{proof}
The last piece of machinery needed is the following result of \citet{BE2013}.

\begin{lemma}[\cite{BE2013}, Theorem 1] Let $g$ be a complete Riemannian metric without conjugate points on the plane $\mathbb R^2$. Then, for every $p\in \mathbb R^2$, the area $|B(p, t)|$ of the metric ball with center $p$ and radius $r$ satisfies
\[
\liminf_{r\to\infty} \frac{|B(p, t)|}{\pi r^2} \geq 1
\]
with equality if and only if $g$ is flat.
\end{lemma}

\begin{proof}[Proof of \ref{thm:Main}]
We shall use $M(\infty)$ to give estimates for the volume growth of the surface $M$. This is reasonable since we have the following from \cite{Har1964,Shio1985}.
\[
\lim_{t\to \infty}\frac{|\partial B(p, t)|}{2\pi t}=\lim_{t\to \infty}\frac{|B(p, t)|}{\pi r^2}
\]
Clearly from \eqref{eq:Ideal=Circle} we see that the inner metric spaces $(S(t), d_t/t)$  converges to $M(\infty)$ as $t\to \infty$ in the Gromov-Hausdorff distance. From that, we obtain
\[
\lim_{t\to \infty}\frac{|B(p, t)|}{\pi r^2}=1.
\]
Now, by the rigidity part of \cite{BE2013}, $M$ is isometric to the flat $\mathbb R^{2}$.
\end{proof}

\section{A digression into Euclid's fifth postulate}
\label{sect:Euclid}
In a Euclidean plane, the traditional formulation of Euclid's fifth postulate usually adopts the form of Playfair's axiom. However, the original formulation of Euclid was as follows:

\begin{quote}(\textbf{Euclid's fifth postulate})
If a line segment intersects two straight lines forming two interior angles on the same side that add to less than two right angles, then the two lines, if extended indefinitely, meet on that side on which the angles add to less than two right angles.
\end{quote}

We will denote by (E) to refer to this formulation, while keeping (PA) for Playfair's axiom.

In the Euclidean plane, both formulations are entirely equivalent, and therefore there is no need to use Euclid's slightly more cumbersome formulation.
Nonetheless, it is natural to reconsider the difference between (PA) and (E) for general Riemannian planes without conjugate points. Our next result shows that the conclusion of  Theorem \ref{thm:Main} holds with (E) instead of (PA) even without the requirement that $M$ admits total curvature.

\begin{thm}\label{thm:Euclid original}
Let $M$ be a Riemannian plane without conjugate points, which admits total curvature, and satisfying condition (E). Then $M$ is isometric to the flat Euclidean plane $\mathbb R^2$.
\end{thm}
\begin{proof}
Proposition 1.7 in \cite{BK1991} shows that the lack of conjugate points is equivalent to the condition that
 given any geodesic $\gamma$ and any point $p$ not on
$\gamma$, there is a geodesic $\beta$ with $\beta(0)=p$ that does not intersect $\gamma$.

Axiom (E) states that any two lines making angles with respect to a common transversal with sum less than two right angles are not parallel. It follows then that parallels are exactly those lines that make equal alternate angles. The standard argument of applying Gauss-Bonnet on arbitrary small parallelograms now yields the claim.
\end{proof}


\bibliographystyle{plainnat}

\bibliography{PA}

\end{document}